\documentclass[11pt]{amsart}
\usepackage{amsmath,amssymb,amsfonts}
\usepackage{bbm}
\usepackage{enumerate}
\usepackage{amscd, latexsym, extpfeil}
\usepackage[all]{xy}
\usepackage{pb-diagram}
\usepackage{mathrsfs}
\usepackage{pictexwd,dcpic}
\usepackage{hyperref,url}
\usepackage{color}

\usepackage{pb-diagram}
\theoremstyle{plain}
\newtheorem{Thm}[equation]{Theorem}

\newtheorem{Cor}[equation]{Corollary}
\newtheorem{Prop}[equation]{Proposition}
\newtheorem{Lem}[equation]{Lemma}
\numberwithin{equation}{section}

\theoremstyle{remark}
\newtheorem{Def}[equation]{Definition}
\newtheorem{Rmk}[equation]{Remark}

\newcommand{\Hom}{\operatorname{Hom}}
\newcommand{\Ext}{\operatorname{Ext}}

\newcommand{\soc}{\operatorname{soc}}
\newcommand{\Gal}{\operatorname{Gal}}
\newcommand{\GL}{\operatorname{GL}}

\newcommand{\Sp}{\operatorname{Sp}}

\newcommand{\Ind}{\operatorname{Ind}}
\newcommand{\Jac}{\operatorname{Jac}}

\newcommand{\id}{\operatorname{id}}

\newcommand{\C}{\mathbb C}

\newcommand{\Z}{\mathbb{Z}}
\newcommand{\R}{\mathbb{R}}

\newcommand{\bm}{\begin{multline*}}
\newcommand{\tu}{\end  {multline*}}

\def\calC{\mathcal{C}}

\def\calF{\mathcal{F}}

\def\calR{\mathcal{R}}

\def\lra{\longrightarrow}

\def\wt{\widetilde}

\DeclareFontFamily{U}{mathx}{\hyphenchar\font45}
\DeclareFontShape{U}{mathx}{m}{n}{
      <5> <6> <7> <8> <9> <10>
      <10.95> <12> <14.4> <17.28> <20.74> <24.88>
      mathx10
      }{}
\DeclareSymbolFont{mathx}{U}{mathx}{m}{n}
\DeclareMathSymbol{\bigtimes}{1}{mathx}{"91}
\title{Big Theta equals small theta generically}
\author{Rui Chen} 

\address{Institute for Advanced Study in Mathmatics, ZheJiang University, East No.7 Building, Zijingang Campus, Hangzhou 310058, China }
\email{rchenmat@zju.edu.cn}

\author{Jialiang Zou}

\address{Department of Mathematics, University of Michigan, 530 Church St, Ann Arbor, MI 48109, U.S.A. }
\email{jlzou@umich.edu}
\begin{document}

\begin{abstract}
In this paper we consider the theta correspondence over a non-Archimedean local field. Using the homological method and the theory of derivatives, we show that under a mild condition the big theta lift is irreducible.
\end{abstract}

\maketitle

\section{Introduction}
The theory of theta correspondence plays an important role in the Langlands program. Roughly speaking, given a reductive dual pair $G\times H$, one can attach a Weil representation $\Omega$ to it. The goal of the (local) theory of theta correspondence is to study the branching of $\Omega$ as a $G\times H$ representation, or more precisely, to study the maximal $\pi$-isotypic quotient (as a representation of $H$)
\[
    \Theta(\pi)
\]
of $\Omega$ for all irreducible smooth representation $\pi$ of $G$. This representation $\Theta(\pi)$ of $H$ is called the ``\textit{big theta lift of $\pi$}''. The celebrated Howe duality conjecture (now a theorem \cite{MR1159105} \cite{MR3502978} \cite{MR3454380}) asserts that if $\Theta(\pi)$ is non-zero then it admits a unique irreducible quotient, which is usually denoted by $\theta(\pi)$ and is called the ``\textit{small theta lift of $\pi$}''. The map
\[
    \pi\longmapsto \theta(\pi)
\]
enjoys many good properties, for example it is injective and compatible with $\gamma$-factors and Plancherel measures. Thanks to the work of Atobe-Gan \cite{MR3714507}, Baki\'{c}-Hanzer \cite{MR4279104} and many others, now we have a complete description of $\theta(\pi)$ in terms of $\pi$ in non-Archimedean case.\\

However in many applications of theta correspondence, instead of the small theta lift one needs to deal with the big theta lift. So it is important to know that in favorable cases, whether we have
\[
    \Theta(\pi) \overset{\text{?}}{=} \theta(\pi).
\]
It is expected that this equality holds ``generically'', and indeed one knows that:
\begin{itemize}
\item in non-Archimedean case, the equality holds if: 
\begin{itemize}
\item[-] $\pi$ is tempered and $G\times H$ is an (almost) equal rank dual pair with $G$ to be the smaller member (see \cite[App. C]{MR3166215});

\vskip 5pt

\item[-] $\pi$ is tempered, satisfying some technical conditions (in terms of its L-parameters), and $G$ is the bigger member in the dual pair $G\times H$ (see \cite[Prop. 5.4]{MR3714507});
\end{itemize}

\vskip 5pt

\item in Archimedean case, the equality holds if $\pi$ is unitary and $G\times H$ is a stable range dual pair with $G$ to be the smaller member (see \cite[Thm. A]{MR3305312}).

\vskip 5pt
\end{itemize}
Besides these and some cases in type II (e.g. \cite[Thm. 1.7]{MR3803804}), it seems to the authors that we only have very limited knowledges on the problem of $\Theta$ vs $\theta$.\\

In this paper we follow an idea of Adams-Prasad-Savin \cite{MR3753906} to attack this problem in the non-Archimedean case. The key point is that the contragredient of the big theta lift can be interpreted functorially as 
\[
    \Theta(\pi)^\vee \simeq \Hom_G(\Omega,\pi)_{sm},
\]
where the subscript ``$sm$'' means taking $H$-smooth vectors. By considering the derived functors 
\[
    \Ext_G^i(\Omega,-)_{sm},
\]
we first show that the functor $\Hom_G(\Omega,-)_{sm}$ is exact when restricted to a subcategory $\calR^{\dagger}(G)$ consisting of finite length representations of $G$ satisfying a mild condition on \textit{factors} (Proposition \ref{P:Ext-Vanish}). Then we use the theory of derivatives to analyze $\Theta(\pi)$ and show its irreducibility (Theorem \ref{T:Main} and Theorem \ref{T:Main2}). As an application we deduce a non-Archimedean analogue of Loke-Ma's result (Corollary \ref{T:AnalogLokeMa}).

\vskip 10pt

\section{Local theta correspondence}\label{S:Theta}
In this section we fix some notations and recall some basic facts of theta correspondence. Let $F$ be a non-Archimedean local field of characteristic $0$, and $E$ either $F$ itself or a quadratic field extension of $F$. Accordingly set 
\[
    c=\begin{cases}
        \id \quad & \textit{if $E=F$};\\
        \textit{the non-trivial element in }\Gal(E/F) \quad & \textit{if $E\neq F$}. 
    \end{cases}
\]
Fix $\epsilon\in\{\pm1\}$. Let $W$ be an $n$-dimensional $(-\epsilon)$-Hermitian space over $E$. Likewise, let $V$ be an $m$-dimensional $\epsilon$-Hermitian space over $E$. We set:
\[
    \ell = n-m+\epsilon_0,
\]
where
\[
    \epsilon_0=\begin{cases}
        \epsilon \quad & \textit{if $E=F$};\\
        0 \quad & \textit{if $E\ne F$}.
    \end{cases}
\]
%In this paper we shall always work under the condition that $\ell\leq 0$, i.e. $G(W)$ is the smaller member of the dual pair. 
Let:
\[
    G(W) = \begin{cases}
        \textit{the metaplectic group }{\rm Mp}(W), \quad & \textit{if $E=F$, $\epsilon=+1$ and $m$ is odd};\\
        \textit{the isometry group of }W, \quad & \textit{otherwise}.
    \end{cases}
\]
When $W=0$, we set $G(W)$ to be the trivial group. We define the group $H(V)$ similarly by switching the roles of $W$ and $V$. Then $G(W)$ and $H(V)$ form a reductive dual pair. To consider the theta correspondence for the reductive dual pair $G(W)\times H(V)$, one requires some additional data:
\begin{itemize}
\item a non-trivial additive character $\psi_F$ of $F$;
\vskip 5pt

\item a pair of characters $\chi_V$ and $\chi_W$ of $E^\times$ satisfying certain conditions (see \cite[Sect. 3.2]{MR3166215} for more details).
\vskip 5pt
%\item a trace zero element $\delta\in E^\times$.
\end{itemize}
To elaborate, the tensor product $V\otimes W$ has a natural symplectic form, which induces a natural map
\begin{equation*}
G(W)\times H(V)\longrightarrow \Sp(W\otimes V).
\end{equation*}
One has the $\C^\times$-cover $\widetilde{\Sp}(W\otimes V)$ of $\Sp(W\otimes V)$, and the character $\psi_F$ determines a Weil representation $\Omega_{\psi_F}$ of $\widetilde{\Sp}(W\otimes V)$. The datum $(\psi_F,\chi_V,\chi_W)$ then allows one to specify a splitting of the cover over $G(W)\times H(V)$. Hence, we have a Weil representation $\Omega=\Omega_{V,W,\psi_F}$ of $G(W)\times H(V)$. 

\vskip 5pt

\begin{Rmk}
One special case is worth noting: when $W$ is a split $2$-dimensional orthogonal space, the group $G(W)$ has a unique (standard) parabolic subgroup $P$, namely the one stabilizing an isotropic line in $W$. By our convention on the isometry group of zero space, the Levi subgroup of $P$ should be understood as $\GL_1$, and indeed one can show that under this convention there is no supercuspidal representation of $G(W)$.
\end{Rmk}

\vskip 5pt

Given an irreducible representation $\pi$ of $G(W)$, the maximal $\pi$-isotypic quotient of $\Omega$ is of the form 
\begin{equation*}
\Theta_{V,W,\psi_F}(\pi)\boxtimes\pi
\end{equation*}
for some smooth representation $\Theta_{V,W,\psi_F}(\pi)$ of $H(V)$ of finite length. By the Howe duality \cite{MR1159105} \cite{MR3502978} \cite{MR3454380}, we have:
\begin{itemize}
\item the maximal semi-simple quotient $\theta_{V,W,\psi_F}(\pi)$ of $\Theta_{V,W,\psi_F}(\pi)$ is irreducible if $\Theta_{V,W,\psi_F}(\pi)$ is non-zero;
\item let $\pi_1$ and $\pi_2$ be irreducible smooth representations of $G(W)$, such that both $\theta_{V,W,\psi_F}(\pi_1)$ and $\theta_{V,W,\psi_F}(\pi_2)$ are non-zero. If $\pi_1\not\simeq\pi_2$, then $\theta_{V,W,\psi_F}(\pi_1)\not\simeq\theta_{V,W,\psi_F}(\pi_2)$.
\end{itemize}
\vskip 5pt
From now on we fix a non-trivial additive character $\psi_F$ and a pair of splitting characters $(\chi_V,\chi_W)$ once for all, and omit them from subscripts of various notations. As we have explicated in the introduction, by definition one has
\[
    \Theta_{V,W}(\pi)^\vee \simeq \Hom_{G(W)}\left(\Omega_{V,W},\pi\right)_{sm}
\]
for any irreducible smooth representation $\pi$ of $G(W)$. This gives a functorial interpretation of the big theta lift. 

\vskip 10pt

\section{The theory of derivatives}

In this section we briefly recall some results from the theory of derivatives for classical groups. For simplicity we only state results for $G(W)$, but all results in this section also hold %(or at least expected to hold) 
for $H(V)$. Firstly we fix some notations and conventions. 
\begin{itemize}
\item Let $\nu$ be the absolute value of $E^\times$, and also regarded as a character of each general linear group $\GL_k=\GL_{k}(E)$ through the determinant map. 

\vskip 5pt

\item The set of all irreducible supercuspidal representation of $\GL_k$ will be denoted by $\calC(\GL_k)$, and by $\calC_{unit}(\GL_k)$ the subset consisting of irreducible unitary supercuspidal representations. 

\vskip 5pt

\item We shall use $\wt{\GL}_k$ to denote the two fold cover of $\GL_k$ as in \cite[Sect. 2.5]{MR3166215}. Using the additive character $\psi_F$ one can define a genuine character $\chi_{\psi_F}$ of $\wt{\GL}_k$ (see \cite[Sect. 2.6]{MR3166215}), and a bijection
\[
    \tau\longmapsto \tau\chi_{\psi_F}
\]
between the set of equivalence classes of smooth representations of $\GL_k$ and the set of equivalence classes of smooth genuine representations of $\wt{\GL}_k$.

\vskip 5pt

\item If $G$ is a general linear group $\GL_k(E)$ or a classical group as in the setting of Section \ref{S:Theta}, we shall write $\mathsf{R}(G)$ for the category of all %smooth (and genuine if $G={\rm Mp}(W)$) representations of $G$
\[
    \begin{cases}
    \textit{smooth genuine reprsentations of }G, \quad & \textit{if }G={\rm Mp}(W);\\
    \textit{smooth representations of }G, \quad & \textit{otherwise}.
    \end{cases}
\] 
The subcategory of $\mathsf{R}(G)$ consisting of finite length objects will be denoted by $\calR(G)$.

\vskip 5pt

\item For a finite length representation $\Pi\in\calR(G)$, we denote by $\soc(\Pi)$ the socle of $\Pi$, namely the maximal semi-simple subrepresentation of $\Pi$; $\Pi$ is said to be \textit{socle irreducible} (``SI'' for short), if $\soc(\Pi)$ is irreducible, and occurs in the semi-simplification $s.s.\Pi$ of $\Pi$ with multiplicity one.

\vskip 5pt

\item Let $\tau$, $\tau_1$ and $\tau_2$ be smooth representations of general linear groups, and $\pi$ smooth representation of a classical group as in Section \ref{S:Theta}. We shall use 
\[
    \tau_1\times \tau_2 \quad \textit{and} \quad \tau\rtimes\pi
\]
to denote the normalized parabolic induction for general linear groups and classical groups. In the case of metaplectic groups, the symbol $\tau\rtimes \pi$ should be understood as the normalized parabolic induction  
\[
    \Ind_P^G\left(\tau\chi_{\psi_F}\boxtimes \pi\right),
\]
as in \cite[Sect. 2.6]{MR3166215}. If $\rho$ is a smooth representation of a general linear group, we also put
\[
    \rho^k \coloneqq \underbrace{\rho\times\cdots\times\rho}_{k\textit{-times}}.
\]
\end{itemize}

Let $\pi\in\calR(G(W))$. For an irreducible supercuspidal representation $\rho\in\calC(\GL_d)$ (where $d\geq 1$) and a positive integer $k\in\Z_{>0}$, the $k$-th $\rho$-derivative $D_\rho^{(k)}\pi$ is defined as follows. Let $P$ be the maximal parabolic subgroup of $G(W)$ with Levi component $\GL_{dk}\times G(W_0)$ (or $\wt{\GL}_{dk}\times_{\mu_2} G(W_0)$ if $G(W)$ is a metaplectic group), where $W_0\subset W$ is a non-degenerate subspace of proper dimension such that $W/W_0$ is split. If the semi-simplified Jacquet module 
\[
    s.s.\Jac_P\pi = \sum_i\tau_i\boxtimes\sigma_i \quad (\textit{or, } s.s.\Jac_P\pi = \sum_i\tau_i\chi_{\psi_F}\boxtimes\sigma_i \textit{ if $G(W)$ is metaplectic}~)
\]
for some irreducible representations $\tau_i$ and $\sigma_i$ of $\GL_{dk}$ and $G(W_0)$, then we define
\[
    D_\rho^{(k)}\pi = \sum_{\{i:\tau_i\simeq \rho^k\}} \sigma_i.
\]
When $G$ does not have such standard parabolic subgroup $P$, we interpret $D_\rho^{(k)}\pi$ to be $0$. If $D_\rho^{(k)}\pi \neq 0$ but $D_\rho^{(k+1)}\pi = 0$, we say that $D_\rho^{(k)}\pi$ is the highest $\rho$-derivative; if $D_\rho^{(1)}\pi = 0$, we say that $\pi$ is $\rho$-reduced. The theory of $\rho$-derivatives has a huge difference between the non conjugate self-dual case and the conjugate self-dual case. We consider these cases separately.\\

\noindent\underline{\textbf{Non conjugate self-dual case:}} when $\rho$ is non conjugate self-dual, the theory is easy. We have the following result. Readers may consult to \cite[Lem. 3.1.3]{MR3268853} and \cite[Sect. 3.2]{MR4549708} for more details.%; and we shall also provide a proof for metaplectic groups in Appendix \ref{App:DerMeta}.

\begin{Thm}\label{T:DerNonSelfDual}
Let $\pi$ be an irreducible smooth representation of $G(W)$.
\begin{enumerate}
\item If $D_\rho^{(k)}\pi$ is the highest $\rho$-derivative of $\pi$, then $D_\rho^{(k)}\pi$ is irreducible, and
\[
    \pi\xhookrightarrow{\quad}\rho^k\rtimes D_\rho^{(k)}\pi.
\]

\item For any positive integer $k$, the induced representation $\rho^k\rtimes\pi$ is SI.
\end{enumerate}
\end{Thm}

\noindent\underline{\textbf{Conjugate self-dual case:}} when $\rho$ is conjugate self-dual, instead of the original $\rho$-derivative, Atobe-M\'inguez defined two new derivatives. For a segment $\{a,a-1,\cdots,b\}\subset \R$,
denote by $\Delta_\rho[a,b]$ the unique irreducible subrepresentation of 
\[
    \rho\nu^a\times\cdots\times\rho\nu^b,
\]
and $Z_\rho[b,a]$ the unique irreducible quotient of the same induced representation. Let $\pi\in\calR(G(W))$. The $\Delta_\rho[0,-1]$-derivative and $Z_\rho[0,1]$-derivative are defined similar to the original derivative: let $P$ be the maximal parabolic subgroup of $G(W)$ with Levi component $\GL_{2dk}\times G(W_0)$ (or $\wt{\GL}_{2dk}\times_{\mu_2} G(W_0)$ if $G(W)$ is a metaplectic group), where $W_0\subset W$ is a non-degenerate subspace of proper dimension such that $W/W_0$ is split. If the semi-simplified Jacquet module 
\[
    s.s.\Jac_P\pi = \sum_i\tau_i\boxtimes\sigma_i \quad (\textit{or, } s.s.\Jac_P\pi = \sum_i\tau_i\chi_{\psi_F}\boxtimes\sigma_i \textit{ if $G(W)$ is metaplectic}~)
\]
for some irreducible representations $\tau_i$ and $\sigma_i$ of $\GL_{2dk}$ and $G(W_0)$, then we define
\[
    D_{\Delta_\rho[0,-1]}^{(k)}\pi = \sum_{\{i:\tau_i\simeq \Delta_\rho[0,-1]^k\}} \sigma_i ,\quad\textit{and} \quad D_{Z_\rho[0,1]}^{(k)}\pi = \sum_{\{i':\tau_{i'}\simeq Z_\rho[0,1]^k\}} \sigma_{i'}.
\]
We can also define the notion of highest derivative and the property of being reduced with respect to these new derivatives. The analogue of Theorem \ref{T:DerNonSelfDual} still holds (see \cite[Lem. 3.5, Prop. 3.7]{MR4549708}).

\begin{Thm}\label{T:DerSelfDual}
Let $\pi$ be an irreducible smooth representation of $G(W)$. Suppose that $\pi$ is $\rho\nu^{-1}$-reduced. (respectively $\rho\nu^1$-reduced).
\begin{enumerate}
\item The highest $\Delta_\rho[0,-1]$-derivative $D_{\Delta_\rho[0,-1]}^{(k)}\pi$ (respectively the highest ${Z_\rho[0,1]}$-derivative $D_{Z_\rho[0,1]}^{(k)}\pi$) of $\pi$ is irreducible, $\rho\nu^{-1}$-reduced (respectively $\rho\nu^1$-reduced), and
\[
    \pi\xhookrightarrow{\quad}{\Delta_\rho[0,-1]}^k\rtimes D_{\Delta_\rho[0,-1]}^{(k)}\pi \quad (\textit{respectively}\quad\pi\xhookrightarrow{\quad}{Z_\rho[0,1]}^k\rtimes D_{Z_\rho[0,1]}^{(k)}\pi).
\]

\item For any positive integer $k$, the induced representation ${\Delta_\rho[0,-1]}^k\rtimes\pi$ (respectively ${Z_\rho[0,1]}^k\rtimes\pi$) is SI.
\end{enumerate}
\end{Thm}

\vskip 5pt

We will also need the following easy corollary of \cite[Prop. 3.8]{MR4549708}, which gives a sufficient condition for the existence of $\Delta_\rho[0,-1]$-derivative.

\begin{Prop}\label{P:AM-Prop3.8}
Let $\pi$ be an irreducible smooth representation of $G(W)$. Suppose that:
\begin{itemize}
\item $\pi$ is non-tempered, and
\vskip 5pt

\item $\pi$ does not have any non conjugate self-dual derivative.
\end{itemize}
\vskip 5pt
Then there exists a \textit{conjugate self-dual} $\rho\in\calC_{unit}(\GL_{d})$, such that the highest $\Delta_{\rho}[0,-1]$-derivative $\pi_0=D_{\Delta_{\rho}[0,-1]}^{(k)}\pi$ is non-zero with $k\geq 1$. 
\end{Prop}

\vskip 5pt

\begin{Rmk}
Although Atobe-M\'inguez only defined these derivatives for symplectic group and split odd special orthogonal groups, these two results also holds for other classical groups (including metaplectic groups, thanks to the work \cite{MR2564837}); their proofs still work with very minor modifications. %For orthogonal and unitary groups (not necessarily quasi-split), their proof works without changing a word; for metaplectic groups, we provide a proof in Appendix \ref{App:DerMeta}.
\end{Rmk}

\vskip 10pt

\section{Factors of a representation}

In this section we single out a subcategory 
\[
    \calR^{\dagger}_V(G(W))
\]
of $\calR(G(W))$. The key property of this subcategory we shall make use of is that $\calR^{\dagger}_V(G(W))$ is ``stable'' under the parabolic induction. Firstly we recall the following Lemma-Definition due to M{\oe}glin-Tadi\'c \cite[Lem. 3.3]{MR1896238}.

\begin{Def}\label{D:Tadic}
Let $\pi$ be an irreducible representation of $G(W)$. The following three sets of irreducible supercuspidal representation $\tau$ of general linear groups coincide.

\begin{enumerate}
\item The set of all $\tau$ for which there exists an irreducible subquotient
\[
    \sigma\boxtimes\pi_{cusp} \quad (\textit{or, }\sigma\chi_{\psi_F}\boxtimes\pi_{cusp} \textit{ if $G(W)$ is metaplectic}~)
\]
of a Jacquet module of $\pi$ (with respect to a standard maximal parabolic subgroup), where $\sigma$ is a representation of a general linear group and $\pi_{cusp}$ is a supercuspidal representation of an appropriate classical group, such that $\tau$ is in the cuspidal support of $\sigma$.

\vskip 5pt

\item The set of all $\tau$ for which there exists an irreducible supercuspidal subquotient
\[
    \rho_1\boxtimes\cdots\boxtimes\rho_r\boxtimes\pi_{cusp} \quad (\textit{or, }\rho_1\chi_{\psi_F}\boxtimes\cdots\boxtimes\rho_r\chi_{\psi_F}\boxtimes\pi_{cusp} \textit{ if $G(W)$ is metaplectic}~)
\]
of a Jacquet module of $\pi$ (with respect to a standard parabolic subgroup), where $\rho_i$ are supercuspidal representations of general linear groups and $\pi_{cusp}$ is a representation of an appropriate classical group, such that $\tau\simeq\rho_i$ for some index $i\in\{1,\cdots,r\}$.

\vskip 5pt

\item The set of all $\tau$ for which there exists irreducible supercuspidal representations $\rho_1,\cdots,\rho_r$ of general linear groups and $\pi_0$ irreducible representation of an appropriate classical group, such that
\[
    \pi\xhookrightarrow{\quad} \rho_1\times\cdots\times\rho_r\rtimes \pi_0,
\]
and $\tau\simeq \rho_i$ for some index $i\in\{1,\cdots,r\}$.
\end{enumerate}
\vskip 5pt
Irreducible supercuspidal representations of general linear groups characterized by one of these descriptions will be called \textit{factors} of $\pi$, and we shall denote this set by $\calF(\pi)$.

\end{Def}

\vskip 5pt

The following property of the set of factors is worth noting. Suppose that $\pi$ is an irreducible subquotient of a representation of the form
\begin{equation*}%\label{E:CuspExpo}\tag{$\dagger$}
    \rho_1\times\cdots\times \rho_r\rtimes\pi_{cusp},
\end{equation*}
where $\rho_i\in\calC(\GL_{d_i})$ and $\pi_{cusp}$ is an irreducible supercuspidal representation of $G(W_c)$ for some non-degenerate subspace $W_c\subset W$ of proper codimension such that $W/W_c$ is split. Then we have
\[
    \calF(\pi)\subset\left\{\rho_1,\left(\rho_1^c\right)^\vee,\cdots,\rho_r,\left(\rho_r^c\right)^\vee\right\}.
\]
Moreover, for any index $i\in\{1,\cdots,r\}$, at least one element of $\left\{\rho_i,\left(\rho_i^c\right)^\vee\right\}$ is in the set $\calF(\pi)$. Here for an irreducible representation $\rho$ of $\GL_d=\GL_d(E)$, we use $\rho^c$ to denote the $c$-conjugation of $\rho$. Readers may consult \cite[Sect. 2]{MR3711838} for more details.

\vskip 5pt

\begin{Def} 
We set $\calR^{\dagger}_V(G(W))$ to be the full subcategory of $\calR(G(W))$ consisting of all finite length representation $\Pi$ of $G(W)$, such that for any irreducible subquotient $\pi$ of $\Pi$ we have
\begin{equation*}%\label{D:condition}\tag{$\spadesuit$}
    \chi_V\nu^{\frac{1-\ell}{2}}, \chi_V\nu^{-\frac{1-\ell}{2}}\not\in\calF(\pi).
\end{equation*}
\end{Def}

\vskip 5pt

As suggested by the notation, this subcategory is not intrinsic and depends on the space $V$. %When there is no confusion, we shall suppress $V$ from the subscript. 
From the definition and the property of $\calF(\pi)$ we can easily deduce the following result.

\begin{Lem}\label{L:BoundExpo}
Consider the parabolically induced representation of $G(W)$:
\[
    \Pi=\rho_1\times\cdots\times \rho_r\rtimes \pi_0,
\]
where $\rho_i\in\calC(\GL_{d_i})$ and $\pi_0$ is an irreducible smooth representation of $G(W_0)$ for some non-degenerate subspace $W_0\subset W$ of proper codimension such that $W/W_0$ is split. Let $V_0\subset V$ be a non-degenerate subspace of the same codimension as $W_0\subset W$ such that $V/V_0$ is split. Then the following are equivalent.
\begin{enumerate}
\item There is an irreducible subquotient $\pi$ of $\Pi$, such that $\pi\in\calR^\dagger_V(G(W))$.

\vskip 5pt

\item For any irreducible subquotient $\pi$ of $\Pi$, we have $\pi\in\calR^\dagger_V(G(W))$.

\vskip 5pt

\item We have 
\[
    \chi_V\nu^{\frac{1-\ell}{2}}, \chi_V\nu^{-\frac{1-\ell}{2}}\not\in \left\{\rho_1,\left(\rho_1^c\right)^\vee,\cdots,\rho_r,\left(\rho_r^c\right)^\vee\right\},
\]
and $\pi_0\in\calR^{\dagger}_{V_0}(G(W_0))$. 
\end{enumerate}

\end{Lem}

\vskip 5pt

\begin{proof}
Trivially $(2)\Rightarrow(1)$. We first show that $(3)\Rightarrow(2)$. Suppose that $\pi_0$ is an irreducible subquotient of 
\begin{equation*}\label{E:4.1}\tag{4.1}
    \tau_1\times \cdots \times \tau_s\rtimes \pi_{cusp},
\end{equation*}
where $\tau_i\in\calC(\GL_{d_i})$ and $\pi_{cusp}$ is an irreducible supercuspidal representation of $G(W_c)$ for some non-degenerate subspace $W_c\subset W_0$ of proper codimension such that $W_0/W_c$ is split. Then any irreducible subquotient $\pi$ of $\Pi$ is also an irreducible subquotient of 
\begin{equation*}\label{E:4.2}\tag{4.2}
    \rho_1\times\cdots\times \rho_r\times\tau_1\times\cdots\times \tau_s\rtimes\pi_{cusp}.
\end{equation*}
As we have explicated, the set $\calF(\pi)$ is contained in 
\[
    \left\{\rho_1,\left(\rho_1^c\right)^\vee,\cdots,\rho_r,\left(\rho_r^c\right)^\vee\right\}\cup\left\{\tau_1,\left(\tau_1^c\right)^\vee,\cdots,\tau_s,\left(\tau_s^c\right)^\vee\right\}.
\]
Hence to show $(2)$ it suffices to show that $\chi_V\nu^{\frac{1-\ell}{2}}$ and $\chi_V\nu^{-\frac{1-\ell}{2}}$ are not contained in the above set. On the one hand, in $(3)$ we have already assumed that
\[
    \chi_V\nu^{\frac{1-\ell}{2}}, \chi_V\nu^{-\frac{1-\ell}{2}}\not\in \left\{\rho_1,\left(\rho_1^c\right)^\vee,\cdots,\rho_r,\left(\rho_r^c\right)^\vee\right\}.
\]
On the other hand, it follows from $\pi_0\in\calR^\dagger_{V_0}(G(W_0))$ that
\[
    \chi_V\nu^{\frac{1-\ell}{2}}, \chi_V\nu^{-\frac{1-\ell}{2}}\not\in \left\{\tau_1,\left(\tau_1^c\right)^\vee,\cdots,\tau_s,\left(\tau_s^c\right)^\vee\right\}.
\]
These together imply that $\chi_V\nu^{\frac{1-\ell}{2}}$ and $\chi_V\nu^{-\frac{1-\ell}{2}}$ are not contained in $\calF(\pi)$, or equivalently $\pi\in\calR^\dagger_V(G(W))$.\\

Next we prove that $(1)\Rightarrow(3)$. Again, suppose that $\pi_0$ is an irreducible subquotient of the induced representation as in (\ref{E:4.1}). Then $\pi$ is an irreducible subquotient of the induced representation as in (\ref{E:4.2}). 
By the property of $\calF(\pi)$, we deduce from $\pi\in\calR^\dagger_V(G(W))$ that
\[
\chi_V\nu^{\frac{1-\ell}{2}}, \chi_V\nu^{-\frac{1-\ell}{2}}\not\in \left\{\rho_1,\left(\rho_1^c\right)^\vee,\cdots,\rho_r,\left(\rho_r^c\right)^\vee\right\}\cup\left\{\tau_1,\left(\tau_1^c\right)^\vee,\cdots,\tau_s,\left(\tau_s^c\right)^\vee\right\}.
\]
This implies the statements in $(3)$.

\end{proof}

\vskip 5pt

\begin{Rmk}
Combining this lemma with the persistience principle \cite[Prop. 4.1]{kudla1996notes}, one can show that for any irreducible smooth representation $\pi\in\calR^\dagger_V(G(W))$, either: 
\begin{itemize}
\item the big theta lift $\Theta_{V,W}(\pi)=0$; or

\vskip 5pt

\item $\pi$ is the first occurence of $\theta_{V,W}(\pi)$ in the Witt tower containing $W$.
\end{itemize}
\end{Rmk}

\vskip 10pt

\section{Ext-theta vanishing for representations without critical factors}

As we have explicated in the introduction, the contragredient of the big theta lift can be intepreted functorially by considering the functor
\[
    \Hom_{G(W)}\left(\Omega_{V,W},-\right)_{sm}: \mathsf{R}\left(G(W)\right) \lra \mathsf{R}\left(H(V)\right).
\]
Our goal in this section is to show that the restriction of this functor $\Hom_{G(W)}\left(\Omega_{V,W},-\right)_{sm}$ to the subcategory $\calR^{\dagger}_V(G(W))$ is exact. To be more precise, we consider the derived functors 
\[
    \Ext^i_{G(W)}\left(\Omega_{V,W},-\right)_{sm}: \mathsf{R}\left(G(W)\right) \lra \mathsf{R}\left(H(V)\right)
\]
of $\Hom_{G(W)}\left(\Omega_{V,W},-\right)_{sm}$ and prove the following Ext-vanishing result.

\begin{Prop}\label{P:Ext-Vanish}
Suppose that $\pi\in \calR^{\dagger}_V(G(W))$. Then we have
\[
    \Ext_{G(W)}^i\left(\Omega_{V,W},\pi\right)_{sm} = 0 \quad \textit{for all $i>0$}.
\]
\end{Prop}

\vskip 5pt

We need the following Ext-version of \cite[Prop. 5.2]{MR3714507}. The proof is totally the same as Atobe-Gan's proof.
%We set $\St_k$ to be the Steinberg representation of $\GL_{k/E}$. 
Let $\pi$ be a smooth representation of $G(W)$.

\begin{Lem}\label{L:KudlaFilCompute}
Suppose that $\pi\simeq \chi_V\tau\rtimes\pi_0$, where $\tau$ is either:
\begin{itemize}
\item an irreducible supercuspidal $\rho\in\calC(\GL_d)$; or

\vskip 5pt

\item the generalized Steinberg representation $\Delta_\rho[0,-1]$ or $\Delta_\rho[1,0]$ for some conjugate self-dual $\rho\in\calC_{unit}(\GL_d)$; or

\vskip 5pt

\item the generalized Speh representation $Z_\rho[0,1]$ or $Z_\rho[-1,0]$ for some conjugate self-dual $\rho\in\calC_{unit}(\GL_d)$,
\end{itemize}
\vskip 5pt
and $\pi_0$ is a smooth representation of $G(W_0)$ for some non-degenerate subspace $W_0\subset W$ of proper codimension such that $W/W_0$ is split. If $\chi_V\nu^{\frac{1-\ell}{2}}$ and $\chi_V\nu^{-\frac{1-\ell}{2}}$ are not in the cuspidal support of $\chi_V\tau$, then for any $i\geq 0$ we have
\[
    \Ext_{G(W)}^i\left(\Omega_{V,W},\chi_V\tau\rtimes \pi_0\right)_{sm} \simeq \chi_W^c\tau^c\rtimes \Ext_{G(W_0)}^i\left(\Omega_{V_0,W_0},\pi_0\right)_{sm}.
\]
Here $V_0\subset V$ is a non-degenerate subspace of the same codimension as $W_0\subset W$ such that $V/V_0$ is split, and we interpret the RHS to be $0$ if there is no such $V_0$.
\end{Lem}

\vskip 5pt

Now back to the proof of Proposition \ref{P:Ext-Vanish}. We shall argue by induction on the Witt index $r_W$ of $W$. In the case that $r_W=0$, $G(W)$ is compact so there is nothing to prove. \\

Suppose that the proposition has been proved when the Witt index $r_W<r$. Now we prove the case when $r_W=r$. The first thing to note is that for any smooth representation $\pi$ of $G(W)$, one has
\[
    \Ext_{G(W)}^i\left(\Omega_{V,W},\pi\right)_{sm} = 0 \quad \textit{for all $i>r$}.
\]
Indeed this follows from the general result \cite[Pg. 98, Sect. 4.2]{BNote}. Next we appeal to a descending argument.

\begin{Lem}\label{L:DescentDegree}
In the context of above, let $i$ be a positive integer. Suppose that 
\[
    \Ext_{G(W)}^{i+1}(\Omega_{V,W},\pi)_{sm}=0
\]
for any $\pi\in \calR^{\dagger}_V(G(W))$. Then we have
\[
    \Ext_{G(W)}^{i}(\Omega_{V,W},\pi)_{sm}=0
\]
for any $\pi\in \calR^{\dagger}_V(G(W))$. 
\end{Lem}

\vskip 5pt

\begin{proof}
It suffices to show the assertion for all irreducible representation $\pi\in \calR^{\dagger}_V(G(W))$. If $\pi$ is supercuspidal, then the assertion holds automatically since $\pi$ is injective. So we can assume that $\pi$ is not supercuspidal. Then there exists some $\rho\in\calC(\GL_{d})$ and an irreducible representation $\pi_0$ of $G(W_0)$ for some non-degenerate subspace $W_0\subset W$ of codimension $2d$, such that one has an epimorphism
\[
    \rho\rtimes\pi_0 \xtwoheadrightarrow{} \pi.
\]
Let $K$ be the kernel of this epimorphism. We complete it to a short exact sequence
\[
    0\lra K \lra \rho\rtimes\pi_0 \longrightarrow \pi\lra 0.
\]
It follows from Lemma \ref{L:BoundExpo} that $\pi_0\in\calR^{\dagger}_{V_0}(G(W_0))$ and $K\in\calR^{\dagger}_V(G(W))$. Applying the functor $\Hom_{G(W)}\left(\Omega_{V,W},-\right)_{sm}$ to this short exact sequence, we obtain from the long exact sequence that
\[
    \Ext^i_{G(W)}\left(\Omega_{V,W},\rho\rtimes\pi_0\right)_{sm} \lra \Ext^i_{G(W)}\left(\Omega_{V,W},\pi\right)_{sm} \lra \Ext^{i+1}_{G(W)}\left(\Omega_{V,W},K\right)_{sm}  
\]
is exact. Now we investigate these Ext spaces. By Lemma \ref{L:KudlaFilCompute} and the induction hypothesis, we have
\[
    \Ext^i_{G(W)}\left(\Omega_{V,W},\rho\rtimes\pi_0\right)_{sm} \simeq \chi_V\chi_W^c\rho^c\rtimes\Ext^i_{G(W_0)}\left(\Omega_{V_0,W_0},\pi_0\right)_{sm} = 0.
\]
On the other hand, by the assumption 
\[
    \Ext^{i+1}_{G(W)}\left(\Omega_{V,W},K\right)_{sm} =0.
\]
Hence we deduce that
\[
    \Ext^i_{G(W)}\left(\Omega_{V,W},\pi\right)_{sm} = 0
\]
as desired. 

\end{proof}

\vskip 5pt

This lemma completes the proof of Proposition \ref{P:Ext-Vanish}.

\vskip 10pt

\section{Irreducibility of the big theta lift}

Our result in the previous section implies that the restriction of the functor 
\[
    \Hom_{G(W)}\left(\Omega_{V,W},-\right)_{sm}: \calR^{\dagger}_V\left(G(W)\right) \lra \calR\left(H(V)\right)
\]
is exact. Now we shall use this to show the following irreducibility result of the big theta lift. 

\begin{Thm}\label{T:Main}
Suppose that $\ell\neq -3$. Then for any irreducible smooth representation $\pi\in\calR^{\dagger}_V(G(W))$, the big theta lift $\Theta_{V,W}(\pi)$ is irreducible if it is non-zero.
\end{Thm}

\vskip 5pt

\begin{proof}
Again we shall prove this theorem by induction on the Witt index $r_W$ of $W$. In the case that $r_W=0$, $G(W)$ is compact so there is nothing to prove. \\

Suppose that the theorem has been proved when the Witt index $r_W<r$. Now we prove the case when $r_W=r$. Let $\pi\in\calR^{\dagger}_V(G(W))$ be an irreducible smooth representation. If $\pi$ is supercuspidal, then by a well-known result of Kudla \cite[Thm. 6.1(1), 6.2(1)]{kudla1996notes} one knows that $\Theta_{V,W}(\pi)$ is irreducible if it is non-zero. So we can assume that $\pi$ is not supercuspidal. There are $3$ cases. \\

$\bullet$ \textit{Case 1}: there exists a \textit{non conjugate self-dual} $\rho\in\calC(\GL_{d})$ such that $D_{\chi_V\rho}^{(1)}\pi\neq 0$. Then there exists an irreducible representation $\pi_0$ of $G(W_0)$ for some non-degenerate subspace $W_0\subset W$ of codimension $2d$, such that one has a monomorphism
\begin{equation*}\label{E:6.1}\tag{6.1}
   \pi\xhookrightarrow{\quad} \chi_V\rho\rtimes\pi_0.
\end{equation*}
By Lemma \ref{L:BoundExpo} we have $\chi_V\rho\neq \chi_V\nu^{\frac{1-\ell}{2}}$ and $\pi_0\in\calR^\dagger_{V_0}(G(W_0))$. Applying the functor $\Hom_{G(W)}\left(\Omega_{V,W},-\right)_{sm}$ to this monomorphism, it follows from Lemma \ref{L:KudlaFilCompute} that
\[
    \Theta_{V,W}(\pi)^\vee\xhookrightarrow{\quad} \chi_W^c\rho^c\rtimes \Theta_{V_0,W_0}(\pi_0)^\vee.
\]
Then by the induction hypothesis $\Theta_{V_0,W_0}(\pi_0)^\vee$ is irreducible. Since $\chi_W^c\rho^c$ is not conjugate self-dual, by Theorem \ref{T:DerNonSelfDual} the induced representation $\chi_W^c\rho^c\rtimes \Theta_{V_0,W_0}(\pi_0)^\vee$ is SI. The Howe duality implies that
\[
    \theta_{V,W}(\pi)^\vee = \soc\left(\chi_W^c\rho^c\rtimes \Theta_{V_0,W_0}(\pi_0)^\vee\right).
\]

\vskip 5pt

On the other hand, applying both the MVW-involution and contragredient functor to the monomorphism (\ref{E:6.1}) (see \cite[Lem. 2.2]{MR3714507}), we get an epimorphism
\[
    \chi_V\rho'\rtimes\pi_0 \xtwoheadrightarrow{} \pi,
\]
where $\rho' = \left(\rho^c\right)^\vee$. Applying the functor $\Hom_{G(W)}\left(\Omega_{V,W},-\right)_{sm}$ to this epimorphism, it follows from Lemma \ref{L:KudlaFilCompute} and Proposition \ref{P:Ext-Vanish} that we have
\[
    \chi_W^c\left(\rho'\right)^c\rtimes \Theta_{V_0,W_0}(\pi_0)^\vee\xtwoheadrightarrow{}\Theta_{V,W}(\pi)^\vee,
\]
hence playing the same trick of MVW-contragredient \cite[Lem. 2.2]{MR3714507} again we obtain an embedding
\[
    \Theta_{V,W}(\pi)^{MVW}\xhookrightarrow{\quad} \chi_W^c\rho^c\rtimes \Theta_{V_0,W_0}(\pi_0)^\vee.
\]
From this one can see that $\Theta_{V,W}(\pi)$ must be irreducible, otherwise $\chi_W^c\rho^c\rtimes \Theta_{V_0,W_0}(\pi_0)^\vee$ will have two irreducible subquotients isomorphic to the its socle, contradicting the property of being SI. \\

$\bullet$ \textit{Case 2}: $\pi$ is not in \textit{Case 1}, and $\pi$ is non-tempered. In this case by Proposition \ref{P:AM-Prop3.8}, there exists a \textit{conjugate self-dual} $\rho\in\calC_{unit}(\GL_{d})$, such that the highest $\Delta_{\chi_V\rho}[0,-1]$-derivative $\pi_0=D_{\Delta_{\chi_V\rho}[0,-1]}^{(k)}\pi$ is non-zero with $k\geq 1$. So there is a monomorphism
\begin{equation*}\label{E:6.2}\tag{6.2}
   \pi\xhookrightarrow{\quad} \left(\chi_V\Delta_{\rho}[0,-1]\right)^k\rtimes\pi_0.
\end{equation*}

\vskip 5pt

The rest of the proof in this case is similar to \textit{Case 1}. Applying the functor $\Hom_{G(W)}\left(\Omega_{V,W},-\right)_{sm}$ to this monomorphism we get
\[
    \Theta_{V,W}(\pi)^\vee\xhookrightarrow{\quad} \Hom_{G(W)}\left(\Omega_{V,W},\left(\chi_V\Delta_{\rho}[0,-1]\right)^k\rtimes\pi_0\right)_{sm}.
\]
Since $\pi\in\calR^{\dagger}_V(G(W))$, by Lemma \ref{L:BoundExpo} and Lemma \ref{L:KudlaFilCompute} we have
\[
    \Hom_{G(W)}\left(\Omega_{V,W},\left(\chi_V\Delta_{\rho}[0,-1]\right)^k\rtimes\pi_0\right)_{sm} \simeq \left(\chi_W^c\Delta_{\rho^c}[0,-1]\right)^k\rtimes \Theta_{V_0,W_0}(\pi_0)^\vee.
\]
Recall that $\pi$ is not in \textit{Case 1}, in particular $\pi$ is $\chi_V\rho\nu^{-1}$-reduced. By Theorem \ref{T:DerSelfDual}, the highest $\Delta_{\chi_V\rho}[0,-1]$-derivative $\pi_0 = D_{\Delta_{\chi_V\rho}[0,-1]}^{(k)}\pi$ is also $\chi_V\rho\nu^{-1}$-reduced. Now we claim that $\Theta_{V_0,W_0}(\pi_0)^\vee$ is $\chi_W^c\rho^c\nu^{-1}$-reduced. With this claim at hand, we deduce from Theorem \ref{T:DerSelfDual} that the induced representation $\left(\chi_W^c\Delta_{\rho^c}[0,-1]\right)^k\rtimes \Theta_{V_0,W_0}(\pi_0)^\vee$ is SI. Hence the Howe duality implies that
\[
    \theta_{V,W}(\pi)^\vee=\soc\left(\left(\chi_W^c\Delta_{\rho^c}[0,-1]\right)^k\rtimes \Theta_{V_0,W_0}(\pi_0)^\vee\right).
\]
On the other hand applying both the MVW-involution and contragredient functor to the monomorphism (\ref{E:6.2}), we get an epimorphism
\[
    \left(\chi_V\Delta_{\rho}[1,0]\right)^{k}\rtimes\pi_0 \xtwoheadrightarrow{} \pi.
\]
Applying the functor $\Hom_{G(W)}\left(\Omega_{V,W},-\right)_{sm}$ to this epimorphism, it follows from Proposition \ref{P:Ext-Vanish} that we have
\[
    \left(\chi_W^c\Delta_{\rho^c}[1,0]\right)^k\rtimes \Theta_{V_0,W_0}(\pi_0)^\vee\xtwoheadrightarrow{} \Theta_{V,W}(\pi)^\vee,
\]
and hence
\[
    \Theta_{V,W}(\pi)^{MVW} \xhookrightarrow{\quad} \left(\chi_W^c\Delta_{\rho^c}[0,-1]\right)^k\rtimes \Theta_{V_0,W_0}(\pi_0)^\vee.
\]
This implies that $\Theta_{V,W}(\pi)$ must be irreducible.

\vskip 5pt

So it only remains to show our claim. We shall prove it by contradiction. Suppose on the contrary that $\Theta_{V_0,W_0}(\pi_0)^\vee$ is not $\chi_W^c\rho^c\nu^{-1}$-reduced. Then there is a monomorphism
\[
    \Theta_{V_0,W_0}(\pi_0) \xhookrightarrow{\quad} \chi_W\rho\nu^{-1}\rtimes \sigma_0
\]
for some irreducible smooth representation $\sigma_0$ of an appropriate classical group. Here we have made use of the fact that $\Theta_{V_0,W_0}(\pi_0)^\vee$ is irreducible. Then
\[
    \pi_0^\vee \xhookrightarrow{\quad} \Hom_{H(V_0)}\left(\Omega_{V_0,W_0}, \chi_W\rho\nu^{-1}\rtimes \sigma_0\right)_{G(W_0)-sm},
\]
where the subscript ``$G(W_0)-sm$'' means taking $G(W_0)$-smooth vectors. Similar to \cite[Prop. 5.2]{MR3714507}, since $\ell\neq -3$ we have
\[
    \Hom_{H(V_0)}\left(\Omega_{V_0,W_0}, \chi_W\rho\nu^{-1}\rtimes \sigma_0\right)_{G(W_0)-sm}\xhookrightarrow{\quad} \chi_V^c\rho^c\nu^{-1}\rtimes \Sigma_0
\]
for some finite length smooth representation $\Sigma_0$ of an appropriate classical group. Combining these monomorphisms and applying the MVW-involution we get
\[
    \pi_0  \xhookrightarrow{\quad} \chi_V\rho\nu^{-1}\rtimes \Sigma_0^{MVW}.
\]
This contradicts to the fact that $\pi_0$ is $\chi_V\rho\nu^{-1}$-reduced. Hence $\Theta_{V_0,W_0}(\pi_0)^\vee$ must be $\chi_W^c\rho^c\nu^{-1}$-reduced. The proof of this claim is the only place that we need the condition $\ell\neq -3$. \\

$\bullet$ \textit{Case 3}: $\pi$ is not in \textit{Case 1}, and $\pi$ is tempered. In this case we can find a conjugate self-dual irreducible supercuspidal representation $\rho\in\calC_{unit}(\GL_d)$ and an irreducible tempered representation $\pi_0$ of some appropriate group, such that 
\[
    \pi \xhookrightarrow{\quad} \chi_V\rho\rtimes\pi_0
\]
as a direct summand. Applying the functor $\Hom_{G(W)}\left(\Omega_{V,W},-\right)_{sm}$ to this embedding, by Lemma \ref{L:KudlaFilCompute} we have
\[
    \Theta_{V,W}(\pi)^\vee \xhookrightarrow{\quad} \chi_W^c\rho^c\rtimes\Theta_{V_0,W_0}(\pi_0)^\vee.
\]
By the induction hypothesis $\Theta_{V_0,W_0}(\pi_0)^\vee$ is irreducible.  

\vskip 5pt

Now note that $\pi_0$ is tempered. 
\begin{itemize}
\item[-] If $\ell\leq 0$, by \cite[Prop. 11.5]{MR3279536} and \cite[Cor. 3.4]{MR1001840} the theta lift $\theta_{V_0,W_0}(\pi_0)$ can be realized using the integration of matrix coefficients, and it is equipped with an $H(V_0)$-invariant Hermitian form. Moreover according to \cite[Thm. A.5]{MR2906912} this Hermitian form on $\theta_{V_0,W_0}(\pi_0)$ is positive definite. Hence $\theta_{V_0,W_0}(\pi_0)$ is unitary. 

\vskip 5pt

\item[-] If $\ell>0$, by \cite[Thm. 4.3]{MR3714507} the theta lift $\theta_{V_0,W_0}(\pi_0)$ of $\pi_0$ is tempered, hence is unitary as well. %Here we have made use of the fact that $\pi_0\in\calR^\dagger_{V_0}(G(W_0))$ to check the conditions in Atobe-Gan's result.
\end{itemize}
\vskip 5pt
The upshot is that, the unitarity of $\theta_{V_0,W_0}(\pi_0)$ implies that $\chi_W^c\rho^c\rtimes\Theta_{V_0,W_0}(\pi_0)^\vee$ is unitary and semi-simple. Thus $\Theta_{V,W}(\pi)$ is also irreducible and unitary. \\

These complete the proof of this theorem.

\end{proof}

%\vskip 5pt

%\begin{Rmk}
%Here is a caveat: M{\oe}glin's result \cite[Thm. 5.2]{MR2906916} is for the symplectic-orthogonal dual pair. If one assumes M{\oe}glin's explicit construction of A-packets for other classical groups (both quasi-split and non quasi-split), then her proof should also work.
%\end{Rmk}

\vskip 10pt

\section{Application: the case of stable range}

The condition that $\ell\neq -3$ in Theorem \ref{T:Main} is very weird. Unfortunately at this point we do not know how to remove this condition, however in the case of stable range we can still show the irreducibility of the big theta lift even if $\ell=-3$.\\ %The key point is using the \textit{Zelevinsky-Aubert duality} to avoid the critical exponent. \\

We first list out all stable range reductive dual pairs $(G(W),H(V))$ such that $\ell=-3$:

\begin{enumerate}
\item $E\neq F$, $\dim W\leq 3$, and $\dim V = \dim W+3$;

\vskip 5pt

\item $E=F$, $\dim W=2$ or $0$, and $\dim V = \dim W + 3+\epsilon_0$.
\end{enumerate}
\vskip 5pt
In all of these cases, the Witt index $r_W$ of $W$ is at most $1$, and so is the $F$-rank of $G(W)$. Now we are ready to state and prove the main result of this section.

\begin{Thm}\label{T:Main2}
Let $(G(W),H(V))$ be a stable range reductive dual pair with $\ell=-3$. Then for any irreducible smooth representation $\pi\in\calR^{\dagger}_V(G(W))$, the big theta lift $\Theta_{V,W}(\pi)$ is irreducible.
\end{Thm}

\vskip 5pt

\begin{proof}
The structure of the proof is the same as Theorem \ref{T:Main}: consider the Witt index $r_W$ of $W$. The basic case (i.e. $r_W=0$) and the supercuspidal case is easy. So now assume that $r_W=1$ and $\pi$ is not supercuspidal. There are $2$ cases. \\

$\bullet$ \textit{Case A}: there exists a \textit{non conjugate self-dual} $\rho\in\calC(\GL_{d})$ such that $D_{\chi_V\rho}^{(1)}\pi\neq 0$. The proof of this case is the same as \textit{Case 1} of Theorem \ref{T:Main}.\\

$\bullet$ \textit{Case B}: $\pi$ is not in \textit{Case A}. In this case we can find a conjugate self-dual $\rho\in\calC_{unit}(\GL_d)$ and an irreducible smooth representation $\pi_0$ of $G(W_0)$ for some non-degenerate subspace $W_0\subset W$ of appropriate codimension, such that there is an monomorphism 
\[
    \pi \xhookrightarrow{\quad} \chi_V\rho\rtimes\pi_0.
\]
Now note that $r_W=1$, this implies that $W_0$ is anisotropic. Hence $G(W_0)$ is compact and $\pi_0$ is unitary. From the unitarity of $\pi_0$ we deduce that $\pi$ is indeed a direct summand of $\chi_V\rho\rtimes\pi_0$. 

\vskip 5pt

The rest of the proof in this case is similar to \textit{Case 3} of Theorem \ref{T:Main}. Applying the functor $\Hom_{G(W)}\left(\Omega_{V,W},-\right)_{sm}$ to the above embedding, by Lemma \ref{L:KudlaFilCompute} we have
\[
    \Theta_{V,W}(\pi)^\vee \xhookrightarrow{\quad} \chi_W^c\rho^c\rtimes\Theta_{V_0,W_0}(\pi_0)^\vee.
\]
Since $G(W_0)$ is compact, $\Theta_{V_0,W_0}(\pi_0)^\vee$ is irreducible, and moreover by \cite[Thm. A]{MR1001840} it is also unitary. Then $\chi_W^c\rho^c\rtimes\Theta_{V_0,W_0}(\pi_0)^\vee$ is unitary and semi-simple, this implies that $\Theta_{V,W}(\pi)$ is also irreducible and unitary. \\

These complete the proof of this theorem.

\end{proof}

\vskip 5pt

\begin{Rmk}
In fact we can prove a stronger result when $\ell=-3$ by using the \textit{Zelevinsky-Aubert duality} and the \textit{Adams' conjecture on the theta lifting} \cite{MR2906916} \cite{bakic2022theta}. However that will make our argument much more technical and complicated. To make the paper simpler and more self-contained, we only deal with the stable range case. 
\end{Rmk}

\vskip 5pt

As an application, we prove the following non-Archimedean analogue of \cite[Thm. A]{MR3305312}.

\begin{Cor}\label{T:AnalogLokeMa}
Suppose that the reductive dual pair $G(W)\times H(V)$ is in the stable range and $G(W)$ is the smaller member. Then for any irreducible unitary representation $\pi$ of $G(W)$, the big theta lift $\Theta_{V,W}(\pi)$ is irreducible.
\end{Cor}

\begin{proof}
Since $G(W)\times H(V)$ is in the stable range, we have
\[
    \frac{1-\ell}{2}\geq \frac{1+n-\epsilon_0}{2}.
\]
Then \cite[Thm. 3.6]{MR3711838} implies that for any irreducible unitary representation $\pi$ of $G(W)$ we have $\pi\in\calR^{\dagger}_V(G(W))$. Therefore by Theorem \ref{T:Main} and Theorem \ref{T:Main2} the big theta lift $\Theta_{V,W}(\pi)$ is irreducible.

\end{proof}

%\appendix

%\section{The theory of derivatives for metaplectic groups}\label{App:DerMeta}

\vskip 15pt

\bibliographystyle{alpha}
%\nocite{*}
\bibliography{ThetavsthetaRef}

\end{document}